\documentclass[10pt,a4paper]{article}
\usepackage{amsmath} 
\usepackage{amssymb} 

\evensidemargin=\oddsidemargin
\evensidemargin=\oddsidemargin\textwidth=133mm \textwidth=149mm
\newtheorem{theorem}{Theorem}[section]
\newtheorem{lemma}{Lemma}[section]

\newenvironment{proof}[1][Proof]{\begin{trivlist}
\item[\hskip \labelsep {\bfseries #1}]}{\end{trivlist}}
\numberwithin{equation}{section} 

\begin{document}
\title {A Derivative-Hilbert operator acting on BMOA space\footnote{   The research
was supported by Zhejiang Province Natural Science Foundation(Grant No. LY23A010003).}}
\author{  Huiling Chen\footnote{E-mail address:  HuillingChen@163.com}\quad\quad Shanli Ye\footnote{Corresponding author.~ E-mail address: slye@zust.edu.cn} \\
(\small \it School of Science, Zhejiang University of Science and Technology,
Hangzhou 310023, China)}
 \date{}
\maketitle

\begin{abstract}
Let $\mu$ be a positive Borel measure on the interval $[0,1)$. The Hankel matrix $\mathcal{H}_{\mu}=(\mu_{n,k})_{n,k\geq 0}$ with entries $\mu_{n,k}=\mu_{n+k}$, where $\mu_{n}=\int_{[0,1)}t^nd\mu(t)$,  induces, formally, the Derivative-Hilbert operator  $$\mathcal{DH}_\mu(f)(z)=\sum_{n=0}^\infty\left(\sum_{k=0}^\infty \mu_{n,k}a_k\right)(n+1)z^n ,  ~z\in \mathbb{D},$$
where $f(z)=\sum_{n=0}^\infty a_nz^n$ is an analytic function in $\mathbb{D}$.  We characterize the  measures  $\mu$ for which $\mathcal{DH}_\mu$ is
 a bounded operator on $BMOA$ space. We also study the analogous problem   from the $\alpha$-Bloch space $\mathcal{B}_\alpha(\alpha>0)$ into the $BMOA$ space.
 \vspace{8pt}
\\
{\small\bf Keywords}\quad {Hilbert operator, Bloch space, BMOA space, Carleson measure
 \\
    {\small\bf 2020 MR Subject Classification }\quad 47B38, 30H30, 30H35\\}
\end{abstract}

\maketitle

\section{\textbf{Introduction}}
\quad
\par
Let $\mathbb{D}=\left \{ z\in \mathbb{C}:\left | z \right |<1  \right \}$ denote respectively the open unit disc and the unit circle in the complex plane $\mathbb{C}$, and let $H(\mathbb{D})$ be the space of all analytic functions in $\mathbb{D}$ and $dA(z)=\frac{1}{\pi}dxdy$ the normalized area Lebesgue measure.

  If $0<r<1$ and $f\in H(\mathbb{D})$, we set
\begin{align}
   & M_p (r,f)=\left( \frac{1}{2\pi} \int_0^{2\pi} |f(re^{i \theta})|^p d\theta \right)^\frac{1}{p}, \quad 0<p<\infty.  \notag\\
   & M_\infty (r,f)=\sup_{|z|=r}|f(z)|.\notag
\end{align}

For $0<p\leq \infty$, the Hardy space $H^p$ consists of those $f \in H(\mathbb{D})$ such that
$$||f||_{H^p} \overset{def}{=} \sup_{0<r<1}M_p(r,f)<\infty.$$

We refer to \cite{5} for the terminology and findings on Hardy spaces.

The space $BMOA$ consists of those functions $f \in H^1$ whose boundary values has bounded mean oscillation on $\partial \mathbb{D}$, in accordance with the definition by John and Nirenberg. Numerous properties and descriptions can be attributed to BMOA functions. Let us mention the following: for $a \in \mathbb{D}$, let $\varphi_a$ be the M$\ddot{o}$bius transformation
defined by $\varphi_a(z)=\frac{a-z}{1-\overline{a}z}$. If $f$ is an analytic function in $\mathbb{D}$, then $f \in BMOA $ if and only if
$$||f||_{BMOA} \overset{def}{=} |f(0)|+||f||_*< \infty,$$
where
$$||f||_* \overset{def}{=}\sup_{a\in\mathbb{D}}\{\int_{\mathbb{D}}|f'(z)|^2(1-|\varphi_a(z)|^2)dA(z)\}^{1/2},$$
  For an exposition on the theory of BMOA functions, one should review the content in \cite{8}.

For $0<\alpha<\infty$, the $\alpha$-Bloch space $\mathcal{B}_{\alpha}$ consists of those functions $f\in H(\mathbb{D})$ with
$$\|f\|_{\mathcal{B}_{\alpha}}= |f(0)|+\sup_{z\in\mathbb{D}}(1-|z|^2)^{\alpha}|f'(z)|<\infty.$$
We can see that $\mathcal{B}_{1}$ is the classical Bloch space $\mathcal{B}$. Consult references \cite{10,14} for the notation and results concerning the Bloch type spaces. It is a recognized fact that $BMOA\varsubsetneq\mathcal{B}$.
\par
Suppose that  $\mu$ is a finite positive Borel measure on $[0,1)$, and the Hankel matrix defined by its elements $\left ( \mu_{n,k} \right ) _{n,k\ge 0}$ with entries $\mu_{n,k}=\mu_{n+k}$, where $\mu_{n}=\int_{[0,1)}^{} t^nd\mu\left ( t \right )$, formally represents the Hilbert operator
$$
\mathcal{H}_\mu(f)(z)=\sum_{n=0}^{\infty}\left ( \sum_{k=0}^{\infty}\mu_{n,k}a_{k}  \right ) z^n,z\in\mathbb{D},
$$
whenever the right hand side is well defined and defines a function in $H(\mathbb{D})$.

The generalized Hilbert operator $\mathcal{H}_\mu$ has been methodically studied in many different spaces, such as Bergman spaces, Bloch spaces, Hardy spaces(e.g.\cite{6,7,9,12}).

In Ye and Zhou's works \cite{16,17}, they defined the derivative-Hilbert operator $\mathcal{DH}_{\mu }$ as follows:
\begin{align}\label{eqn1.1}
\mathcal{DH}_{\mu }\left ( f \right ) \left ( z \right ) =\sum_{n=0}^{\infty} \left (\sum_{k=0}^{\infty} \mu_{n,k}a_k  \right )\left ( n+1 \right ) z^n.
\end{align}
\par  
  Another generalized integral-Hilbert operator $\mathcal{I}_{{\mu}_\alpha }(\alpha\in\mathbb{N}^+)$ relevant to $\mathcal{DH}_{\mu }$ defined by
\begin{align}\label{eqn1.2}
 \mathcal{I}_{{\mu}_\alpha }\left ( f \right ) \left ( z \right ) =\int_{[0,1)}^{} \frac{f\left ( t \right ) }{\left ( 1-tz \right )^\alpha  }d\mu \left ( t \right )
\end{align}
whenever the right hind side is well defined and defines an analytic function in $\mathbb{D}$. If $\alpha=1$, then $\mathcal{I}_{{\mu}_\alpha }$ is the integral operator $\mathcal{I}_{{\mu}}$. Ye and Zhou characterized the measures $\mu$ for which $ \mathcal{I}\mu_2$
and $\mathcal{DH}_\mu$ are bounded (resp., compact) on the Bloch space \cite{16} and on the Bergman spaces
\cite{17}. In this article, we can also gain the operators $\mathcal{DH}_\mu$ and $ \mathcal{I}\mu_2$ are intricately connected.

Let us review the comcept of the Carleson-type measures, which is a useful tool for understanding Banach spaces of analytic functions.

If $I\subset\partial \mathbb{D}$ in an arc, $|I|$ denotes the length of $I$, the Carleson square $S(I)$ is defined as
$$
 S(I)=\left\{z=re^{it}:e^{it}\in I, 1-\frac{|I|}{2\pi}\leq r < 1 \right\}.
$$

Suppose that $\mu$ is a positive Borel measure on $\mathbb{D}$. For $0\leq \beta < \infty$ and $ 0<s< \infty $, we say that $\mu$ is a $\beta$-logarithmic $s$-Carleson measure if there exists a positive constant $C$ such that
 $$\sup_I\frac{\mu(S(I))(\log\frac{ 2\pi }{|I|})^\beta}{|I|^s} \leq C, \quad  \quad I \subset \partial \mathbb{D}.$$
If $\mu(S(I))(\log\frac{ 2\pi }{|I|})^\beta=o(|I|^s)$ as  $|I|\rightarrow 0$, we say that $\mu$ is a vanishing $\beta$-logarithmic $s$-Carleson measure.

  A positive Borel measure on $[0,1)$  can also be seen as a Borel measure on $\mathbb{D}$ by identifying it with the measure $\mu$ defined by
$$\tilde{\mu}(E)=\mu(E\bigcap [0,1))$$
for any Borel subset $E$ of $\mathbb{D}$.  Then we say that $\mu$ is a $\beta$-logarithmic $s$-Carleson measure if there exists a positive constant $C$ such that
$$
\mu ([t,1)) \log^\beta\frac{e}{1-t } \le C(1-t)^s,\quad for\ all \ 0\le t<1.
$$

In detail, $\mu$ is a $s$-Carleson measure if $\beta=0$. If $\mu$ satisfies
$$
\lim_{t\to1^-}\frac{\mu ([t,1)) \log^\beta\frac{e}{1-t }}{(1-t)^s} =0,
$$
we say that $\mu$ is a vanishing $\beta$-logarithmic $s$-Carleson measure(see \cite{13,18}).

In this article we focus on qualify the positive Borel measure $\mu$ such that $\mathcal{DH}_{\mu}$ is bounded on $BMOA$ space. Additionally, we also do similar work for the operators acting from $\mathcal{B}^\alpha(0<\alpha<\infty)$ spaces into $BMOA$ space.

Throughout this work, the symbol $C$ represents an absolute constant which may be different from one occurrence to next. We employ the notation $``A\lesssim B"$ means that there exists a positive constant $C=C(\cdot)$ such that $A\le CB$ and $``A\gtrsim B"$ is interpreted in a comparable fashion.

\section{The operator $\mathcal{DH}_{\mu }$ acting on the $BMOA$ space }
In this section, we shall give a characterization of those measures $\mu$ for which $\mathcal{DH}_{\mu }$ is a bounded operator on the $BMOA$ space. The following two lemmas are easily obtained by the fact that $BMOA\varsubsetneq\mathcal{B}$, \cite[Theorem 2.1]{16} and \cite[Theorem 2.2]{16}.
\begin{lemma}\label{th2.1}
Let $\mu$ be a positive Borel measure on $[0,1)$. Then the following two statements are equivalent.

(i) $\int_{[0,1)}^{} \log \frac{e}{1-t} d\mu(t)<\infty;$

(ii) For any given $f\in BMOA$, the integral in (\ref{eqn1.2}) uniformly converges on any compact subset of $\mathbb{D}.$

\end{lemma}

\begin{lemma}\label{th-2.2}
Let $\mu$ be a positive Borel measure on $[0,1)$ with $\int_{[0,1)}^{}\log\frac{e}{1-t}d\mu(t)<\infty$. If the measure $\mu$ is a $1$-logarithmic $1$-Carleson measure, then for every $f\in BMOA$ (\ref{eqn1.1}) is a well defined analytic function in $\mathbb{D}$ and $\mathcal{DH}_{\mu }= \mathcal{I}_{{\mu}_2 }$.
\end{lemma}

The following Lemma is a characterization of Carleson measure on $[0,1)$.
\begin{lemma}\label{le4.1111}\cite{24}
If $0<\beta<\infty,0\le \alpha<\gamma<\infty$ and $\mu$ is a positive Borel measure on $[0,1)$. Then the following statements are equivalent.

(i) $\mu$ is an $\gamma$-Carleson measure;

(ii)
$$
\sup\limits_{a\in \mathbb{D}}\int_{[0,1) }^{}\frac{\left ( 1-\left | a \right |  \right )^\beta }{\left ( 1-x \right )^\alpha\left ( 1-\left | a \right | x \right )^{\gamma+\beta-\alpha}}d\mu(t)<\infty;
$$

(iii)
$$
\sup\limits_{a\in \mathbb{D}}\int_{[0,1) }^{}\frac{\left ( 1-\left | a \right |  \right )^\beta }{\left ( 1-x \right )^\alpha\left ( 1-ax \right )^{\gamma+\beta-\alpha}}d\mu(t)<\infty.
$$
\end{lemma}

\begin{lemma}\label{le4.3}\cite{22}
If $\gamma>-1,\alpha>0,\beta>0$ with $\alpha+\beta-\gamma-2>0$. Then, for all $a,b\in\mathbb{D}$, we have that
$$
\int_{\mathbb{D}  }^{}\frac{\left ( 1-\left | z \right | ^2 \right )^\gamma}{\left | 1-\bar{a}z  \right |^\alpha\left | 1-\bar{b}z \right |^\beta}dA(z)\lesssim \frac{1}{\left | 1-\bar{a}b  \right | ^{\alpha+\beta-\gamma-2}},\quad \alpha,\beta<2+\gamma;
$$
$$
\int_{\mathbb{D}  }^{}\frac{\left ( 1-\left | z \right | ^2 \right )^\gamma }{\left | 1-\bar{a}z  \right |^\alpha\left | 1-\bar{b}z \right |^\beta}dA(z)\lesssim \frac{\left(1-|a|^2\right)^{2+\gamma-\alpha}}{\left | 1-\bar{a}b  \right | ^\beta},\quad \beta<2+\gamma<\alpha.
$$
\end{lemma}

\begin{theorem}\label{th3.2}
Let $\mu$ be a positive measure on $[0,1)$ which satisfies the condition in Theorem \ref{th-2.2}. Then  $\mathcal{DH}_{\mu}: BMOA\rightarrow BMOA$ is bounded if and only if  $\mu$ is a $1$-logarithmic $2$-Carleson measure.
\end{theorem}

\begin{proof}
Since $BMOA\varsubsetneq\mathcal{B}$ and the fact that
$$
|f(z)|\lesssim\|f\|_\mathcal{B}\log \frac{e}{1-|z|},\quad z\in\mathbb{D}
$$
for any $f\in\mathcal{B}$, we obtain  that
\begin{align}\label{eqn-2.1}
\int_{[0,1)}^{} \left | f(t) \right | d\mu(t)&\lesssim \left \| f \right \| _{\mathcal{B}}\int_{[0,1)}^{} \log \frac{e}{1-t}  d\mu(t)\notag\\
&\lesssim\left \| f \right \| _{BMOA}\int_{[0,1)}^{} \log \frac{e}{1-t}d\mu(t)<\infty, \quad f\in BMOA.
\end{align}
Whenever $0<r<1, f\in BMOA$ and $g\in H^1$, we have that
\begin{align}\label{eqn-3.2}
\int_{0}^{2\pi } \int_{[0,1)}^{} \left | \frac{ f(t)g\left ( e^{i\theta } \right ) }{\left ( 1- re^{i\theta } t \right ) ^2}  \right | d\mu (t)d\theta &\le\frac{1}{\left ( 1-r \right ) ^2}\int_{[0,1)}^{}\left | f(t) \right |d\mu(t)\int_{0}^{2\pi }  \left | g\left ( e^{i\theta } \right )  \right |  d\theta\notag\\
&\lesssim\frac{\left \| g \right \|_{H^1} }{\left ( 1-r \right ) ^2} <\infty .
\end{align}
Using this, together with Fubini's theorem, and Cauchy's integral representation of $H^1$\cite{5}, we conclude that whenever $0<r<1, f\in BMOA$ and $g\in H^1$,
\begin{align}\label{eqn-2.4}
   \frac{1}{2\pi}\int_{0}^{2\pi } \overline{{\mathcal{DH}_{\mu}(f)\left ( re^{i\theta } \right ) }} g\left ( e^{i\theta } \right )d\theta&=\frac{1}{2\pi} \int_{0}^{2\pi } \left ( \int_{[0,1)}^{} \frac{\overline{f(t)}d\mu(t) }{\left ( 1-tre^{-i\theta } \right )^2 }  \right ) g\left ( e^{i\theta } \right )d\theta\notag\\
&=\frac{1}{2\pi}\int_{[0,1)}^{}\overline{f(t)}\int_{0}^{2\pi}\frac{g\left ( e^{i\theta } \right ) }{\left ( 1-tre^{-i\theta }\right )^2}d\theta d\mu(t)\notag\\
&= \int_{[0,1)}^{}\overline{f(t)}{\left ( g\left ( rt \right ) rt \right )}'  d\mu(t)\notag\\
&= \int_{[0,1)}^{}\overline{f(t)}\left ( g\left ( rt \right )+ rt{g}'\left ( rt \right )   \right )  d\mu(t).\
\end{align}

 Recall the duality relation $(H^1)^\ast\cong BMOA$(see \cite{8}), under the pairing
$$
<F,G>=\lim_{r\to1-}\frac{1}{2\pi}\int_{0}^{2\pi}\overline{F\left(re^{i\theta}\right)}G\left(e^{i\theta}\right)d\theta, \quad F\in BMOA,G\in H^1.
$$
This, and using (\ref{eqn-2.4}), it is easy to see that $\mathcal{DH}_{\mu}: BMOA\rightarrow BMOA$ is bounded if and only if
\begin{align}\label{eqn4.4}
\left | \int_{[0,1)}^{}\overline{f(t)}\left ( g\left ( rt \right )+ rt{g}'\left ( rt \right )   \right )  d\mu(t) \right | \lesssim\left \| f \right \|_{BMOA}\left \| g \right \|_{H^1  }, 0\le r<1,f\in BMOA,g\in H^1.
\end{align}

Suppose that $\mathcal{DH}_{\mu}: BMOA\rightarrow BMOA$ is bounded. For $0<b<1$, we set
$$
f_b(z)=\log \frac{e}{1-bz},\quad g_b(z)=\frac{1-b^2}{\left ( 1-bz \right )^2  }\quad z\in\mathbb{D}.
$$
Then $f_b(z)\in BMOA$, $g_b(z)\in H^1$ and
$$
\sup\limits_{ b\in(0,1)}\left \| f_b \right \|_{BMOA }\lesssim1 \quad and \quad \sup\limits_{ b\in(0,1)}\left \| g_b\right \|_{H^1 }\lesssim1.
$$
Then
$$
\begin{aligned}
1&\gtrsim\sup\limits_{ b\in(0,1)}\left \| f_b \right \|_{BMOA}\sup\limits_{ b\in(0,1)}\left \| g_b \right \|_{H^1 }\\
&\gtrsim \left | \int_{[0,1)}^{}\overline{f_b(t)}\left ( g_b\left ( rt \right )+ rt{g}'_b\left ( rt \right )   \right )d\mu(t) \right |\\
&\gtrsim \int_{[b,1)}^{}\log\frac{e}{1-bt}\left (\frac{1-b^2}{\left ( 1-brt \right )^2  }+2br^2t\frac{1-b^2}{\left ( 1-brt \right )^3  }  \right )d\mu(t)\\
&\gtrsim \frac{\log \frac{e}{1-b^2}}{ \left ( 1-b^2 \right ) ^2 }\mu([b,1)).
\end{aligned}
$$
Therefore, we conclude that $\mu$ is a $1$-logarithmic $2$-Carleson measure.

On the contrary, suppose that $\mu$ is a $1$-logarithmic $2$-Carleson measure. Let $\nu$ be the Borel measure on $[0,1)$ defined by $d\nu(t)=\log\frac{e}{1-t}d\mu(t)$, which is a $2$-Carleson measure by Proposition 2.5 in \cite{9}. Take $f(z)=\sum_{k=0}^{\infty}a_kz^k\in BMOA$, for any $z\in\mathbb{D}$, using (\ref{eqn-2.1}) we have that
$$
\begin{aligned}
&\left \| \mathcal{DH}_\mu(f)  \right \| _{BMOA}\\
=&\sup\limits_{ a\in\mathbb{D}}\left ( \int_{\mathbb{D} }^{}\left | \int_{[0,1)}^{} \frac{2tf(t)}{(1-tz)^3}d\mu(t)\right |^2\left ( 1-\left | \varphi _a(z) \right |^2 \right )dA(z)   \right )^\frac{1}{2}\\
\lesssim&\left \| f \right \|_{BMOA }\sup\limits_{ a\in\mathbb{D}}\left ( \int_{\mathbb{D} }^{}\left ( \int_{[0,1)}^{} \frac{\log\frac{e}{1-t}}{\left | 1-tz \right | ^3}d\mu(t)\right )^2\left ( 1-\left | \varphi _a(z) \right |^2 \right ) dA(z)   \right )^\frac{1}{2}.\\
\end{aligned}
$$
Using Minkowski's inequality, Lemma \ref{le4.3} and  Lemma \ref{le4.1111}, we have
$$
\begin{aligned}
&\sup\limits_{ a\in\mathbb{D}}\left ( \int_{\mathbb{D} }^{}\left (\int_{[0,1)}^{} \frac{\log\frac{e}{1-t}}{\left | 1-tz \right | ^3}d\mu(t)\right )^2\left ( 1-\left | \varphi _a(z) \right |^2 \right ) dA(z)   \right )^\frac{1}{2}\\
=&\sup\limits_{ a\in\mathbb{D}}\left ( \int_{\mathbb{D} }^{}\left (\int_{[0,1)}^{} \frac{1}{\left | 1-tz \right | ^3}d\nu(t)\right )^2\left ( 1-\left | \varphi _a(z) \right |^2 \right ) dA(z)   \right )^\frac{1}{2}\\
\lesssim&\sup\limits_{ a\in\mathbb{D}} \int_{[0,1)}^{}\left(\int_{\mathbb{D} }^{}\frac{1}{\left | 1-tz \right | ^6}\left (1-\left | \varphi _a(z) \right |^2 \right )dA(z) \right )^\frac{1}{2}d\nu(t)\\
\lesssim&\left ( 1-\left | a \right |^2  \right ) ^\frac{1}{2}\sup\limits_{ a\in\mathbb{D}} \int_{[0,1)}^{}\left(\int_{\mathbb{D} }^{}\frac{1-\left | z \right |^2  }{\left | 1-tz \right | ^6\left | 1-\bar{a}z  \right | ^2}dA(z) \right )^\frac{1}{2}d\nu(t)\\
\lesssim&\sup\limits_{ a\in\mathbb{D}} \int_{[0,1)}^{}\frac{\left ( 1-\left | a \right |^2  \right ) ^\frac{1}{2}}{\left ( 1-t^2 \right ) ^\frac{3}{2}\left | 1-ta  \right |}d\nu(t)<\infty.
\end{aligned}
$$
Consequently, we deduce that $\mathcal{DH}_{\mu}: BMOA\to BMOA$ is bounded.
\end{proof}

\section{$\mathcal{DH}_{\mu }$ acting from $\mathcal{B}_\alpha$ spaces to BMOA spaces }
\par
In this section, we aim to study the measures $\mu$ for which $\mathcal{DH}_{\mu }$ is a bounded  operator from  $\mathcal{B}_\alpha(\alpha>0)$ into $BMOA$.

\begin{lemma}\label{Lm1.1}\cite{19} 
Suppose that $\alpha>0$. Then the following statements hold:

(i) If $0<\alpha<1$, then $f\in\mathcal{B}_\alpha$ are bounded$;$

(ii) If $\alpha=1$, then $|f(z)|\lesssim\log \frac{e}{1-\left | z \right |}\|f\|_{\mathcal{B}};$

(iii) If $\alpha>1$, then $f\in\mathcal{B}_\alpha$ if and only if $\left | f\left ( z \right )  \right | =O\left((1-|z|^2)^{1-\alpha}\right)$.
\end{lemma}

\begin{lemma}\label{lm4.2}\cite{20}  
Suppose that $0<\alpha<\infty$, and let $\mu$ be a positive Borel measure on $[0,1)$.
\begin{itemize}
			\item [$(i)$] If $0<\alpha<1$, then for any given $f\in\mathcal{B}_\alpha$, the integral  $\mathcal{I}_{{\mu}_2 }\left ( f \right ) \left ( z \right ) =\int_{[0,1)}^{} \frac{f\left ( t \right ) }{\left ( 1-tz \right )^2  }d\mu \left ( t \right )$  defined a well defined analytic function in $\mathbb{D}$ if and only if the measure $\mu$ is finite;

	\item [$(ii)$] If $\alpha=1$, then  for any given $f\in\mathcal{B}_\alpha$, the integral  $\mathcal{I}_{{\mu}_2 }\left ( f \right ) \left ( z \right ) =\int_{[0,1)}^{} \frac{f\left ( t \right ) }{\left ( 1-tz \right )^2  }d\mu \left ( t \right )$  defined a well defined analytic function in $\mathbb{D}$ if and only if the measure satisfies $\int_{[0,1)}\log\frac{e}{1-t}d\mu(t)<\infty;$
			\item [$(iii)$] If $\alpha>1$, then for any given $f\in\mathcal{B}_\alpha$, the integral  $\mathcal{I}_{{\mu}_2 }\left ( f \right ) \left ( z \right ) =\int_{[0,1)}^{} \frac{f\left ( t \right ) }{\left ( 1-tz \right )^2  }d\mu \left ( t \right )$  defined a well defined analytic function in $\mathbb{D}$  if and only if the measure satisfies $\int_{[0,1)}\frac{1}{(1-t)^{\beta-1}}d\mu(t)<\infty.$
		\end{itemize}

\end{lemma}

\begin{lemma}\label{Th2.2}\cite{20}
Suppose $0<\alpha<\infty$ and let $\mu$ be a positive Borel measure on $[0,1)$. Then (\ref{eqn1.1}) is a well defined analytic function in $\mathbb{D}$ and $\mathcal{DH}_{\mu }= \mathcal{I}_{{\mu}_2 }$ for every $f\in\mathcal{B}_\alpha$ in the two following cases:
\par
(i) measure $\mu$ is a $s$-Carleson measure for some $s>0$ if $0<\alpha\le1$;

(ii) measure $\mu$ is an $\alpha$-Carleson measure if $\alpha>1$.
\end{lemma}

\begin{theorem}\label{TH4.1}
Suppose $0<\alpha<1$ and let $\mu$ be a positive measure on $[0,1)$ which satisfies the conditions in Lemma \ref{lm4.2} and Lemma \ref{Th2.2}. Then  $\mathcal{DH}_{\mu}: \mathcal{B}_\alpha\rightarrow BMOA$ is bounded  if and only if  $\mu$ is a $2$-Carleson measure.
\end{theorem}

\begin{proof}
Lemma \ref{Lm1.1} implies that
$$
\int_{[0,1)}^{} \left | f(t) \right | d\mu(t)<\infty, \quad for \ all\ f\in\mathcal{B}^\alpha .
$$
Arguing as in the proof of Theorem \ref{th3.2}, we can say that whenever $0<r<1$, $f\in\mathcal{B}^\alpha$ and $g\in H^1$, we have that
\begin{align}\label{eqn-3.1}
  \frac{1}{2\pi}\int_{0}^{2\pi } \overline{{\mathcal{DH}_{\mu}(f)\left ( re^{i\theta } \right ) }} g\left ( e^{i\theta } \right )d\theta
&=\int_{[0,1)}^{}\overline{f(t)}\left ( g\left ( rt \right )+ rt{g}'\left ( rt \right )   \right )  d\mu(t).\
\end{align}

 Recall the duality relation $(H^1)^\ast\cong BMOA$ and using (\ref{eqn-3.1}), it is easy to see that $\mathcal{DH}_{\mu}: \mathcal{B}_\alpha\rightarrow BMOA$ is bounded if and only if
\begin{align}\label{eqn4.4}
\left | \int_{[0,1)}^{}\overline{f(t)}\left ( g\left ( rt \right )+ rt{g}'\left ( t \right )   \right )  d\mu(t) \right | \lesssim\left \| f \right \|_{\mathcal{B}^\alpha  }\left \| g \right \|_{H^1  }.
\end{align}

Suppose that $\mathcal{DH}_{\mu}: \mathcal{B}_\alpha\rightarrow BMOA$ is bounded. For $0<b<1$, we set that
$$
f_b(z)=1 \quad and \quad g_b(z)=\frac{1-b^2}{\left ( 1-bz \right )^2  }\quad z\in\mathbb{D}.
$$
Then $f_b(z)\in\mathcal{B}_\alpha$, $g_b(z)\in H^1$ and
$$
\sup\limits_{ b\in(0,1)}\left \| f_ b \right \|_{\mathcal{B}^\alpha }\lesssim1 \quad and \quad \sup\limits_{ b\in(0,1)}\left \| g_b \right \|_{H^1 }\lesssim1.
$$
Then
$$
\begin{aligned}
1&\gtrsim\sup\limits_{ b\in(0,1)}\left \| f_b \right \|_{\mathcal{B}^\alpha }\sup\limits_{ b\in(0,1)}\left \| g_b \right \|_{H^1 }\\
&\gtrsim \left | \int_{[0,1)}^{}\overline{f_b(t)}\left ( g_b\left ( rt \right )+ rt{g}'_b\left ( rt \right )   \right )d\mu(t) \right |\\
&\gtrsim\int_{[b,1)}^{}\left (\frac{1-b^2}{\left ( 1-brt \right )^2  }+2br^2t\frac{1-b^2}{\left ( 1-brt \right )^3  }  \right )d\mu(t)\\
&\gtrsim\frac{1}{\left ( 1-b^2 \right )^2  }\mu([b,1)).
\end{aligned}
$$
Therefore, $\mu$ is a $2$-Carleson measure.

 Otherwise, if $\mu$ is a $2$-Carleson measure and $f(z)\in\mathcal{B}_\alpha$, we obtain that
$$
\begin{aligned}
&\left \| \mathcal{DH}_\mu(f)  \right \| _{BMOA}\\
=&\sup\limits_{ a\in\mathbb{D}}\left ( \int_{\mathbb{D} }^{}\left | \int_{[0,1)}^{} \frac{2tf(t)}{(1-tz)^3}d\mu(t)\right |^2\left ( 1-\left | \varphi _a(z) \right |^2 \right )dA(z)   \right )^\frac{1}{2}\\
\lesssim&\left \| f \right \|_{\mathcal{B}_\alpha  }\sup\limits_{ a\in\mathbb{D}}\left ( \int_{\mathbb{D} }^{}\left ( \int_{[0,1)}^{} \frac{1 }{\left | 1-tz \right | ^3}d\mu(t)\right )^2\left ( 1-\left | \varphi _a(z) \right |^2 \right ) dA(z)   \right )^\frac{1}{2}.\\
\end{aligned}
$$
Using Minkowski's inequality, Lemma \ref{le4.3} and  Lemma \ref{le4.1111}, we have that
$$
\begin{aligned}
&\sup\limits_{ a\in\mathbb{D}}\left ( \int_{\mathbb{D} }^{}\left (\int_{[0,1)}^{} \frac{1}{\left | 1-tz \right | ^3}d\mu(t)\right )^2\left ( 1-\left | \varphi _a(z) \right |^2 \right ) dA(z)   \right )^\frac{1}{2}\\
\lesssim&\sup\limits_{ a\in\mathbb{D}} \int_{[0,1)}^{}\left(\int_{\mathbb{D} }^{}\frac{1}{\left | 1-tz \right | ^6}\left (1-\left | \varphi _a(z) \right |^2 \right )dA(z) \right )^\frac{1}{2}d\mu(t)\\
\lesssim&\left ( 1-\left | a \right |^2  \right ) ^\frac{1}{2}\sup\limits_{ a\in\mathbb{D}} \int_{[0,1)}^{}\left(\int_{\mathbb{D} }^{}\frac{1-\left | z \right |^2  }{\left | 1-tz \right | ^6\left | 1-\bar{a}z  \right | ^2}dA(z) \right )^\frac{1}{2}d\mu(t)\\
\lesssim&\int_{[0,1)}^{}\frac{\left ( 1-\left | a \right |^2  \right ) ^\frac{1}{2}}{\left ( 1-t^2 \right ) ^\frac{3}{2}\left | 1-ta  \right |}d\mu(t)<\infty.
\end{aligned}
$$
Consequently, this means that $\mathcal{DH}_{\mu}: \mathcal{B}_\alpha\rightarrow BMOA$ is bounded.
\end{proof}

\begin{theorem}\label{th3.6}
Let $\mu$ be a positive Borel measure on $[0,1)$ with $\int_{[0,1)}^{}\log \frac{e}{1-t}  d\mu(t)<\infty$ and satisfies the Lemma \ref{Th2.2}. Then $\mathcal{DH}_{\mu}: \mathcal{B}\rightarrow BMOA$ is bounded if and only if  $\mu$ is a $1$-logarithmic $2$-Carleson measure.
\end{theorem}
\begin{proof}
Since $\int_{[0,1)}^{}\log \frac{2}{1-t}  d\mu(t)<\infty$, it follows that
$$
\int_{[0,1)}^{}\left | f(t) \right |d\mu(t) \lesssim\left \| f \right \|_\mathcal{B} \int_{[0,1)}^{}\log\frac{e}{1-t}d\mu(t)< \infty,\quad for\ all\ f\in\mathcal{B}.
$$
Hence, by (\ref{eqn-3.1}), we have that
\begin{align}\label{eqn555}
\frac{1}{2\pi}\int_{0}^{2\pi } \overline{{\mathcal{DH}_{\mu}(f)\left ( re^{i\theta } \right ) }} g\left ( e^{i\theta } \right )d\theta=\int_{[0,1)}^{}\overline{f(t)}\left ( g\left ( rt \right )+ rt{g}'\left ( rt \right )   \right )  d\mu(t)\quad 0<r<1,f\in\mathcal{B},g\in H^1.
\end{align}
Using duality, we see that $\mathcal{DH}_{\mu}: \mathcal{B}\rightarrow BMOA$ is bounded if and only if
\begin{align}\label{eqn5555}
\left | \int_{[0,1)}^{}\overline{f(t)}\left ( g\left ( rt \right )+ rt{g}'\left (rt \right )   \right )  d\mu(t) \right |\lesssim\|f\|_{\mathcal{B}}\|g\|_{H^1},f\in\mathcal{B},g\in H^1.
\end{align}
From moment, the proof is similar to the proof of Theorem \ref{th3.2}, we shall omit the details.

\end{proof}

\begin{theorem}\label{th3.8}
Suppose $\alpha>1$ and let $\mu$ be a positive measure on $[0,1)$ which satisfies the conditions in Lemma \ref{lm4.2} and Lemma \ref{Th2.2}. Then  $\mathcal{DH}_{\mu}: \mathcal{B}_\alpha\rightarrow BMOA$ is bounded if and only if $\mu$ is a $(1+\alpha)$-Carleson measure.
\end{theorem}

\begin{proof}
 Suppose that $\mathcal{DH}_{\mu}: \mathcal{B}_\alpha\rightarrow BMOA$ is bounded. For $0<b<1$, set
$$
f_b(z)=\frac{1-b^2}{\left ( 1-bz \right )^{\alpha}  } \quad and\quad g_b(z)=\frac{1-b^2}{\left ( 1-bz \right )^2  },\quad  z\in\mathbb{D}.
$$
Then $f_b(z)\in\mathcal{B}_\alpha$, $g_b(z)\in H^1$ and
$$
\sup\limits_{ b\in(0,1)}\left \| f_b \right \|_{\mathcal{B}_\alpha }\lesssim1 \quad and \quad \sup\limits_{ b\in(0,1)}\left \| g_b \right \|_{H^1 }\lesssim1.
$$
By (\ref{eqn4.4}), we get that
$$
\begin{aligned}
1&\gtrsim\sup\limits_{ b\in(0,1)}\left \| f_b \right \|_{\mathcal{B}^\alpha }\sup\limits_{ b\in(0,1)}\left \| g_b \right \|_{H^1 }\\
&\gtrsim \left | \int_{[0,1)}^{}\overline{f_b(t)}\left ( g_b\left ( rt \right )+ rt{g}'_b\left ( rt \right )   \right )d\mu(t) \right |\\
&\gtrsim\int_{[b,1)}^{}\frac{1-b^2}{\left ( 1-bt \right )^{\alpha}  }\left (\frac{1-b^2}{\left ( 1-brt \right )^2  }+2br^2t\frac{1-b^2}{\left ( 1-brt \right )^3  }  \right )d\mu(t)\\
&\gtrsim\frac{1}{\left ( 1-b^2 \right )^{1+\alpha}  }\mu([b,1)).
\end{aligned}
$$
Therefore, $\mu$ is a $(1+\alpha)$-Carleson measure.

 Otherwise, let $\mu$ be a $(1+\alpha)$-Carleson measure, and let $d\nu(t)=(1-t)^{1-\alpha}d\mu(t)$. Then $\nu$ is a $2$-Carleson measure by Proposition 2.5 in \cite{9}. Then, using Minkowski's inequality, Lemma \ref{le4.3} and  Lemma \ref{le4.1111}, we obtain that
$$
\begin{aligned}
&\left \| \mathcal{DH}_\mu(f)  \right \| _{BMOA}\\
=&\sup\limits_{ a\in\mathbb{D}}\left ( \int_{\mathbb{D} }^{}\left | \int_{[0,1)}^{} \frac{2tf(t)}{(1-tz)^3}d\mu(t)\right |^2\left ( 1-\left | \varphi _a(z) \right |^2 \right )dA(z)   \right )^\frac{1}{2}\\
\lesssim&\left \| f \right \|_{\mathcal{B}_\alpha  }\sup\limits_{ a\in\mathbb{D}}\left ( \int_{\mathbb{D} }^{}\left ( \int_{[0,1)}^{} \frac{ (1-t)^{1-\alpha}}{\left | 1-tz \right | ^3}d\mu(t)\right )^2\left ( 1-\left | \varphi _a(z) \right |^2 \right ) dA(z)   \right )^\frac{1}{2}\\
\lesssim&\left \| f \right \|_{\mathcal{B}_\alpha  }\sup\limits_{ a\in\mathbb{D}}\left ( \int_{\mathbb{D} }^{}\left (\int_{[0,1)}^{} \frac{1}{\left | 1-tz \right | ^3}d\nu(t)\right )^2\left ( 1-\left | \varphi _a(z) \right |^2 \right ) dA(z)   \right )^\frac{1}{2}\\
\lesssim&\left \| f \right \|_{\mathcal{B}_\alpha  }\sup\limits_{ a\in\mathbb{D}}\int_{[0,1)}^{}\frac{\left ( 1-\left | a \right |^2  \right ) ^\frac{1}{2}}{\left ( 1-t^2 \right ) ^\frac{3}{2}\left | 1-ta  \right |}d\mu(t)<\infty.
\end{aligned}
$$
Therefore, $\mathcal{DH}_{\mu}: \mathcal{B}_\alpha\rightarrow BMOA$ is bounded.
\end{proof}

 \end{document}